\documentclass[12pt]{amsart}
\usepackage{graphicx,amssymb,stmaryrd,amsfonts,epsfig,amsthm,a4,amsmath,mathrsfs}
\usepackage{vmargin,url}
\usepackage{eufrak}
\usepackage[latin1]{inputenc}
\psfigdriver{dvips}

\newtheorem{thm}{Theorem}[section]
\newtheorem{cor}[thm]{Corollary}
\newtheorem{lem}[thm]{Lemma}
\newtheorem{prop}[thm]{Proposition}
\theoremstyle{definition}
\newtheorem{defn}[thm]{Definition}

\theoremstyle{remark}
\newtheorem{que}[thm]{Question}
\newtheorem{rem}[thm]{Remark}

\newtheorem{exe}[thm]{Example}

\numberwithin{equation}{section}
\newcommand{\CB}{\textnormal{CB}}
\newcommand{\cb}{\textnormal{cb}}
\newcommand{\Hom}{\textnormal{Hom}}
\begin{document}

\title[Cantor-Bendixson rank of metabelian groups]{On the Cantor-Bendixson rank of metabelian groups}
\author{Yves Cornulier}
\address{IRMAR \\ Campus de Beaulieu \\
35042 Rennes Cedex, France}
\email{yves.decornulier@univ-rennes1.fr}

\date{December 19, 2009}




\begin{abstract}
We study the Cantor-Bendixson rank of metabelian and virtually metabelian groups in the space of marked groups, and in particular, we exhibit a sequence $(G_n)$ of 2-generated, finitely presented, virtually metabelian groups of Cantor-Bendixson rank~$\omega^n$.
\end{abstract}
\maketitle
\section{Introduction}

Let $G$ be a discrete group. Under pointwise convergence, the set $\mathcal{N}(G)$ of normal subgroups is a Hausdorff compact, totally discontinuous space. This topology, sometimes referred to as the {\it Chabauty topology}, was studied in many papers, including \cite{Chab,Gri,Cham,CG,CGP}. If $F_d$ denotes the non-abelian free group on $d$ generators, we can view $\mathcal{N}(F_d)$ as the set $\mathcal{G}_d$ of {\it marked groups on $d$ generators}, through the identification $N\mapsto F_d/N$.

As a topological space, the identification of $\mathcal{G}_d$ seems to be a difficult problem. We focus here on the Cantor-Bendixson rank, which is defined as follows. If $X$ is a topological space, we define its derived subspace $X^{(1)}$ as the subset of accumulation points in $X$. Iterating over ordinals
$$X^{(0)}=X,\;\;X^{(\alpha+1)}=X^{(\alpha)(1)},\quad X^{(\lambda)}=\bigcap_{\beta<\alpha}X^{(\beta)}\textnormal{ for limit }\lambda,$$
we have a non-increasing family $X^{(\alpha)}$ of closed subsets. If $x\in X$, we write $$\CB_X(x)=\sup\{\alpha|x\in X^{(\alpha)}\}$$ if this supremum exists, in which case it is a maximum. Otherwise we say that $x$ is in the {\it condensation part} (or {\it perfect kernel}) of $X$ and we write $\CB_X(x)=\mathfrak{C}$, where the symbol $\mathfrak{C}$ is not an ordinal. If $\CB_X(x)\neq\mathfrak{C}$ for all $x\in X$, i.e.~if $X^{(\alpha)}$ is empty for some ordinal, we say that $X$ is {\it scattered}. If $G$ is a group, we define its (intrinsic) {\it Cantor-Bendixson rank} $\cb(G)$ as $\CB_{\mathcal{N}(G)}(\{1\})$.

Groups $G$ with $\cb(G)=0$, which include finite groups and simple groups, are called {\it finitely discriminable} and were considered in \cite{CGP}. However, most groups, like infinite residually finite groups, are not finitely discriminable. Let us begin by a very simple example (contained in Proposition \ref{hirsch}).

\begin{prop}
Let $G$ be a finitely generated nilpotent group. Then $\cb(G)=h(G)$, the Hirsch length of $G$.
\end{prop}

For instance, we have $\cb(\mathbf{Z}^k)=k$, which was already mentioned, without proof, in \cite[Section 6]{CGP}. The reader can check it as an warm-up exercise; precisely the statement to prove by induction is that if $A$ is a finitely generated abelian group virtually isomorphic to $\mathbf{Z}^k$, then it has Cantor-Bendixson rank $\cb(A)=k$.

So far all known examples either satisfied $\cb(G)<\omega$ or $\cb(G)=\mathfrak{C}$. Our main result is to leap from $\omega$ to $\omega^\omega$.

\begin{thm}
Fix any $d\ge 2$. Then $\mathcal{G}_d$ contains points of Cantor-Bendixson rank equal to any ordinal $\alpha<\omega^\omega$.
More precisely, for every $\alpha<\omega^\omega$, there exists a finitely presented, 2-generated metabelian-by-(finite cyclic) group $H$ with $\cb(H)=\alpha$.\label{thmain}
\end{thm}

The second statement implies the first as for every {\it finitely presented} $d$-generated group $H$, the space $\mathcal{G}_d$ contains $\mathcal{N}(H)$ as a clopen subset. Note that it is known that, on the other hand, as a particular case of \cite[Theorem~3]{Ols}, every non-elementary hyperbolic group $G$ satisfies $\cb(G)=\mathfrak{C}$ (another proof is given by \cite{Cham} when $G$ is torsion-free) and in particular $\cb(F_d)=\mathfrak{C}$.

A pleasant class of groups, for which the study of the Cantor-Bendixson rank can be carried out, is the class of groups satisfying max-n, i.e.~in which there is no infinite increasing sequence of normal subgroups.

\begin{prop}\label{propi}
Let $G$ be a group satisfying max-n. Then the space $\mathcal{N}(G)$ is scattered. If moreover every quotient of $G$ is residually finite, then
we have $\cb(G)=\sup\{\cb(H)+1\}$, where $H$ ranges over quotients groups of $G$ with infinite kernel.
\end{prop}

An important class of groups with max-n is the class of finitely generated, virtually abelian-by-polycyclic groups \cite{Hall}, and these groups are residually finite (as well as their quotients) by a result of Roseblade \cite{Ros}. In particular, this includes finitely generated, virtually metabelian groups, for which however residual finiteness is much easier to obtain \cite{Hall2}.

The gist of Theorem \ref{thmain} is the study of the Cantor-Bendixson rank of finitely generated metabelian groups. However in this case (metabelian instead of virtually metabelian) we have a bound on the exponent in terms of the number of generators. Recall the (standard) wreath product $H\wr G$ refers to the semidirect product $H^{(G)}\rtimes G$.

\begin{thm}\label{thmbounds}
Let $G$ be a finitely generated metabelian group.

\begin{enumerate}
\item\label{in1} Suppose that $G$ sits inside an exact sequence $$1\to M\to G\to Q\to 1,$$ where $M$ is abelian and $Q$ is abelian of $\mathbf{Q}$-rank $\le d$. Then $$\cb(G)<\omega^{d+1}.$$ Moreover, this bound is sharp, as the wreath product $\mathbf{Z}^k\wr\mathbf{Z}^d$ satisfies $$\cb(\mathbf{Z}^k\wr\mathbf{Z}^d)=\omega^{d}\cdot k$$ for $d\ge 0$, $k\ge 1$.
\item\label{melib} If $G$ is $d$-generated and $d\ge 2$, then $$\cb(G)\le\omega^{d}\cdot (d-1),$$ with equality if and only if $G$ is isomorphic to the free metabelian group on $d$ generators.
\item\label{itemsplit} If $G$ is $d$-generated and the above exact sequence is split, then $\cb(G)<\omega^{d}$, and this bound is sharp if $d\ge 2$.
\end{enumerate}
\end{thm}

We actually give a precise computation of $\cb(G)$ for any finitely generated metabelian group. For a general finitely generated metabelian group $G$, we proceed as follows. If $N$ is a normal subgroup of $G$, define the $G$-{\it Hirsch length} $h_G(N)$ as the supremum of lengths $k$ of chains of normal subgroups of $G$ contained in $N$ $$N_0\subset N_1\subset \dots \subset N_k,\quad N_i/N_{i-1}\textnormal{ infinite }\forall i.$$
The {\it Hirsch radical} of $G$ is the largest normal subgroup $\textnormal{Hir}(G)=N$ of $G$ such that $h_G(N)<\infty$. This is well-defined, since $G$ satisfies max-n. 

In Section \ref{length}, we recall the notion of reduced length of modules introduced in \cite{Cor}; this is an ordinal-valued length characterized, when $A$ is a finitely generated commutative ring, by the property, for finitely generated $A$-modules $M$
$$\ell'(M)=\sup\{\ell'(M/N)+1:N\textnormal{ infinite $A$-submodule of }M\}.$$
Notably, if $M$ has Krull dimension $d\ge 0$, then $$\omega^{d-1}\le\ell'(M)<\omega^d$$ (with $\omega^{-1}=0$). Let $W(M)$ denote the largest $A$-submodule of $M$ of Krull dimension $\le 1$.

\begin{thm}\label{cbme}
Let $G$ be a finitely generated metabelian group in an extension
$$1\to M\to G\to Q\to 1,$$
with $M,Q$ abelian, and view $M$ as a $\mathbf{Z}[Q]$-module. Then
\begin{enumerate}
\item\label{icb} $\cb(G)=\ell'(M/W(M))+h_G(\textnormal{Hir}(G))$. 
\item\label{est} In particular, if $d$ is the Krull dimension of the $\mathbf{Z}[Q]$-module $M$, we have
$$\ell'(M)\le\cb(G)<\ell'(M)+\omega.$$
\item\label{pri} If $P$ is a prime ideal in $\mathbf{Z}[Q]$, if $\mathbf{Z}[Q]/P$ has Krull dimension $d\ge 2$, $M$ is isomorphic to a torsion-free $\mathbf{Z}[Q]/P$-module of rank $r$, and the action of $Q$ on $M$ is faithful, then $\cb(G)=\omega^{d-1}\cdot r$.
\end{enumerate}
\end{thm}

To prove the finite presentability of the virtually metabelian groups in Theorem \ref{thmain}, we apply a general criterion due to Bieri and Strebel \cite{BS}. This criterion is explained in Section \ref{bs}. It is important that the groups in Theorem \ref{thmain} are finitely presented: indeed, if $G$ is a group with $d$ given generators, then $\mathcal{N}(G)$ is always closed in $\mathcal{G}_d$, but is open if and only if $G$ is finitely presented (see \cite[Lemma 1.3]{CGP}). Otherwise we can define $\cb^e(G)$ as the Cantor-Bendixson rank of $G$ as an element of $\mathcal{G}_d$. By \cite[Lemma 1]{CGP}, this does not depend on the choice of a finite generating family of $G$. If $G$ is finitely presented then $\cb^e(G)=\cb(G)$. Groups with $\cb^e(G)=0$ are finitely presented and are the main subject of the paper \cite{CGP}. On the other hand, we have the following result, which was asserted without proof in \cite[Section 6]{CGP} in the case of $\mathbf{Z}\wr\mathbf{Z}$. We give here a proof in Section \ref{secw}.

\begin{prop}Let $H,G$ be finitely generated, with $H\neq\{1\}$ and $G$ infinite. Then
$$\cb^e(H\wr G)=\mathfrak{C}.$$\label{wreathcond}
\end{prop}

\begin{que}
It would be interesting to know if any infinitely presented, finitely generated metabelian group satisfies $\cb^e(G)=\mathfrak{C}$. With L.~Guyot, we are able to prove it in the special case of abelian-by-cyclic finitely generated groups, for instance for $\mathbf{Z}[1/6]\rtimes_{2/3}\mathbf{Z}$.
\end{que}

\begin{que}
The bounds given in Theorem \ref{thmbounds} are not optimal when the metabelian group is assumed finitely presented. What are then the optimal bounds? 
\end{que}

Let $X$ be a topological space. The least $\alpha$ such that $X^{(\alpha)}=X^{(\alpha+1)}$ is called the Cantor-Bendixson rank of $X$ and is denoted by $\CB(X)$. This is always well-defined, since the non-increasing chain $(X^{(\alpha)})$ always stabilizes. For instance, $\CB(X)=0$ if and only if $X$ is perfect. In general, $\CB(X)$ is the supremum of $\CB_X(x)+1$, where $x$ ranges over the points not in the condensation part of $X$. By Theorem \ref{thmain}, for every $d\ge 2$, we have $\CB(\mathcal{G}_d)\ge\omega^\omega$. 
If $G$ is a group, and if $\cb(G)\neq\mathfrak{C}$, then $\CB(\mathcal{N}(G))\ge\cb(G)+1$. This is an equality under the assumptions of Proposition \ref{propi}, but not in general: for instance, in \cite[Proof of Theorem~5.3]{CGP}, an isolated group $G$ was given with a normal subgroup $K$ such that $G/K$ is free of rank two. Then since $G$ is isolated, $\cb(G)=0$, but it follows from Theorem \ref{thmain} that $\CB(\mathcal{N}(G))\ge\omega^\omega$, because since $G/K$ is finitely presented, $\mathcal{N}(G)$ contains $\mathcal{N}(G/K)$ as clopen subset.

Last but not least, we can ask

\begin{que}
For $2\le d<\infty$, do we have $\CB(\mathcal{G}_d)>\omega^\omega$?
\end{que}
I do not have any clue how to construct a finitely generated group $G$ with $\cb(G)=\omega^\omega$, although it does probably exist. Since $\mathcal{G}_d$ is compact metrizable, $\CB(\mathcal{G}_d)$ is a countable ordinal. It would be surprising if its value would depend on $d$.

\setcounter{tocdepth}{1}
\tableofcontents
\noindent \textbf{Outline.} The groups referred to in Theorem \ref{thmain} are constructed in Section \ref{sectionvmg}, at the end of which we indicate why they satisfy the claimed property; this relies on results of Sections \ref{length}, \ref{finiteaction}, and \ref{bs}.
The assertions in Proposition \ref{propi} are particular cases of Lemma \ref{maxscat}, Proposition \ref{cblp} and Corollary \ref{ccblp}.
The inequality in Theorem \ref{thmbounds}(\ref{in1}) follows from Theorem \ref{cbme}(\ref{est}), and the example in Theorem \ref{thmbounds}(\ref{in1}) is obtained in Section \ref{se}. Theorem \ref{thmbounds}(\ref{melib}),(\ref{itemsplit}) are obtained in Section \ref{fsme}. Theorem \ref{cbme} is proved in Paragraph \ref{pcb}.

\medskip

Section \ref{sectionvmg} essentially extends the introduction. Sections \ref{se}, \ref{finiteaction}, and Paragraph \ref{splitm} partly rely on Section \ref{length}. At this notable exception, the different sections can be read independently.

\medskip

\noindent{\bf Acknowledgments.} I thank Alexander Kechris, Simon Thomas, Robert Young, and especially Luc Guyot and Pierre de la Harpe for valuable discussions and suggestions.

\section{Examples of finitely presented virtually metabelian groups}\label{sectionvmg}

\subsection{Construction}
Let us describe the groups constructed in the second part of Theorem \ref{thmain}. We postpone all the proofs to Paragraph \ref{cproofs}. The easiest (and most natural) construction provides a 4-generated group; we then explain how to reduce to 3, and then 2 generators.

Fix the integer $d\ge 1$, and $d$ formal variables $(x_i)$, which for convenience we view as indexed by $\mathbf{Z}/d\mathbf{Z}$. Consider the ring $A_d=\mathbf{Z}[(x_i),s^{-1}]$ where $s=\prod_i(x_i-x_i^2)$. Then $\mathbf{Z}^{2d}=\mathbf{Z}^{\mathbf{Z}/d\mathbf{Z}}\times \mathbf{Z}^{\mathbf{Z}/d\mathbf{Z}}$ acts by multiplication on $A_d$, where, if the canonical basis is denoted by $((e_i),(f_i))$, $e_i$ acts by multiplication by $x_i$ and $f_i$ by multiplication by $1-x_i$.

The semidirect product $H_d=A_d\rtimes\mathbf{Z}^{2d}$ has a faithful representation by $2\times 2$-matrices over $A_d$
$$(0,e_i)\mapsto
\begin{pmatrix}
  x_i & 0  \\
  0 & 1  \\
\end{pmatrix};\quad(0,f_i)\mapsto
\begin{pmatrix}
  1-x_i & 0  \\
  0 & 1  \\
\end{pmatrix};\quad u=(1,0)\mapsto
\begin{pmatrix}
  1 & 1  \\
  0 & 1  \\
\end{pmatrix},$$
whose image is the set of all matrices of the form
$$\begin{pmatrix}
  \prod_ix_i^{n_i}(1-x_i)^{m_i} & P  \\
  0 & 1  \\
\end{pmatrix},((n_i),(m_i))\in\mathbf{Z}^{2d},P\in A_d.$$

\begin{prop}
The $(2d+1)$-generated metabelian group $H_d$ is finitely presented, and $$\cb(H_d)=\omega^d.$$
\end{prop}

The finite presentability of $H_d$ is obtained from the computation of the Bieri-Strebel geometric invariant carried out in Section \ref{bs}.

The computation of the Cantor-Bendixson rank essentially relies on the computation of the Cantor-Bendixson rank of the ring $A_d$ (i.e.~of the ideal $\{0\}$ in the set of ideals of $A_d$), which was proved in \cite{Cor} to be equal to $\omega^d$.

Next, we can form the semidirect product $H_d\rtimes\mathbf{Z}/d\mathbf{Z}$, where $\mathbf{Z}/d\mathbf{Z}$ (whose canonical generator we denote by $\sigma$) permutes shifts the variables. This group is virtually metabelian, and is generated by $\{u,e_1,f_1,\sigma\}$. As it contains $H_d$ as a subgroup of finite index, it is finitely presented as well.

\begin{prop}
The $4$-generated virtually metabelian group $\Gamma_d=H_d\rtimes\mathbf{Z}/d\mathbf{Z}$ is finitely presented and satisfies $$\cb(\Gamma_d)=\omega^d.$$
\end{prop}
Now the proof relies on the study of the space of $\mathbf{Z}/d\mathbf{Z}$-invariant ideals in $A_d$. This fits in the context of modules endowed with an action of a finite group, and was not considered in \cite{Cor}, so we prove the necessary preliminaries in Section \ref{finiteaction}.

To pass from 4 to 3 generators, assume that $d$ is odd and replace $\sigma$ by the generator $\gamma$ of $\mathbf{Z}/2d\mathbf{Z}$ which acts on $A_d$ by ring automorphisms, mapping $x_i$ to $1-x_{i+1}$ for all $i\in\mathbf{Z}/d\mathbf{Z}$. In particular $\gamma^d$ sends $x_i$ to $1-x_i$ and $\gamma^{d+1}=\sigma$. So the group $\Gamma'_d$ generated by $\{u,e_1,\gamma\}$ contains $\Gamma_d$ as a subgroup of index 2. For similar reasons, it has Cantor-Bendixson rank $\omega^d$.

Finally, to get a 2-generated group, we consider the subgroup $\Lambda_d$ generated by $\{ue_1,\gamma\}$. Denote by $\Lambda'_d$ the normal subgroup generated by $ue_1$, so that $\Lambda_d=\Lambda'_d\rtimes\langle\gamma\rangle$.

\begin{prop}
The metabelian group $\Lambda'_d$ is finitely presented, as well as $\Lambda_d$. It lies in an extension
$$1\to M\to \Lambda'_d\to Q\to 1,$$
with $Q\simeq\mathbf{Z}^{2d}$, freely generated by the images of all $ue_i$ and $uf_i$, and $M$ is isomorphic as a $\mathbf{Z}[Q]$-module to the kernel of the ring homomorphism of $\mathbf{Z}[Q]$ onto $\mathbf{Z}[1/2]$ mapping all generators to $1/2$. Moreover, we have $$\cb(\Lambda_d)=\cb(\Lambda'_d)=\omega^d.$$
\end{prop}

\subsection{Proofs}\label{cproofs}

The finite presentability of $H_d$ is a consequence of Corollary \ref{fpanneau}, using Bieri-Strebel's characterization of finitely presented metabelian groups (Theorem \ref{bist}). So its overgroups of finite index $\Gamma_d$ and $\Gamma'_d$ are also finitely presented. 
Finally Lemma \ref{bigsubgroup} implies that $\Lambda'_d$ is finitely presented, as well as its overgroup of finite index $\Lambda_d$.

All the groups $G=H_d,\Gamma_d,\Gamma'_d,\Lambda_d,\Lambda'_d$ arise in an extension
$$1\to M\to G\to R\to 1,$$
where $R$ contains $\mathbf{Z}^{2d}$ as a subgroup of finite index, and $M$ is an ideal in $A_d$. For $G=H_d,\Gamma_d,\Gamma'_d$, $M=A_d$; in the two last cases, $M\neq\{0\}$ because $[ue_1,uf_1]\neq 1$. As the Krull dimension of $A_d$ is $d+1$, by Corollary \ref{ideal}, $\ell'_Q(M)=\omega^d$ in all cases. Since $M$ satisfies the hypotheses of Corollary \ref{owncentralizer}, we obtain $\ell'(G)=\ell'_Q(M)=\omega^d$.

\section{Length and Cantor-Bendixson rank}\label{length}

\subsection{Length of noetherian modules}\label{lengthcr}

Let $A$ be a ring (not necessarily commutative). Recall that an $A$-module $M$ has {\it finite} length if there is an upper bound on the length $d$ of increasing chains $$0=M_0\subset M_1\subset\dots \subset M_d=M$$
of $A$-submodules of $M$, and the least bound is called the length of $M$. By a theorem of Jordan and H\"older, modules of finite length can be characterized as modules that are simultaneously noetherian (every increasing chain of submodules stabilizes) and artinian (every decreasing chain stabilizes). However in general, most noetherian modules have infinite length and it is natural to provide a notion of ordinal length. This was first done by Bass \cite{Bass} (in a commutative setting), using well-ordered decreasing chains of submodules. Then Gulliksen \cite{Gull} provided the inductive definition which follows, slightly less intuitive at first sight, but more handy to deal with. 

\begin{defn}
Define inductively, for every noetherian $A$-module, its ordinal length $\ell(M)=\ell_A(M)$ as
$$\ell(M)=\sup\{\ell(M/N)+1:N\textnormal{ nonzero $A$-submodule of $M$}\}.$$
\end{defn}

This has to be viewed as an inductive definition: the starting point is $\ell(\{0\})=\sup(\emptyset)=0$. In more generality, if $X$ is a noetherian partially ordered set, we can define an ordinal-valued function on $X$ by $\ell(x)=\sup\{\ell(y)+1:y>x\}$. The uniqueness of $\ell$ follows from noetherianity of $X$. For the existence, set $M_x=\{y:y\ge x\}$, consider the set $V$ of $u\in X$ such that there exists a function $\ell$ satisfying the inductive condition on $M_u$. If $V\neq X$, then its complement contains a maximal element $u$. So for any $x>u$, the number $\ell(x)$ is uniquely defined. Therefore the inductive definition shows that $\ell$ can be defined on $M_u$, contradicting that $u\notin V$. Here, the partially ordered set is the set of quotients of the module $M$.

If $\alpha$ is a non-zero ordinal, there exists a unique ordinal $\beta$ such that $\omega^\beta\le\alpha<\omega^{\beta+1}$, and we write $\beta=\deg(\alpha)$.

\begin{defn}
Let $M$ be a noetherian $A$-module. The {\it Krull dimension} of $M$ is defined as the ordinal $\deg(\ell(M))$ if $M\neq \{0\}$, and $-1$ if $M=\{0\}$.
\end{defn}
The more usual notion of Krull dimension, used in the non-commutative setting, is called the ``deviation of the poset of submodules of $M$" (see \cite[Chap.~6]{MR}). We do not need this definition, but it is important to mention that it is equivalent to the one given here \cite[Theorem~2.3]{Gull}. Moreover, when $A$ is commutative, it coincides \cite[Chap.~6.4]{MR} with the usual notion of Krull dimension defined in terms of chains of prime ideals, defined inductively as
$$\dim(M)=\sup\{\dim(A/\mathcal{P})+1:\mathcal{P}\textnormal{ non-minimal prime ideal of A/\textnormal{Ann}(M)}\}.$$

Let $A$ again be arbitrary (not necessarily commutative) and let $M$ be a noetherian $A$-module. If $\ell(M)=\omega^\alpha$ for some ordinal $\alpha$, we say that $M$ is {\it critical}, or $\alpha$-{\it critical}. This means that the Krull dimension of $M$ is $\alpha$, but the Krull dimension of any proper quotient of $M$ is $<\alpha$. A {\it critical series} for $M$ is a composition series
$$0=M_0\subset M_1\subset \dots\subset M_k=M,$$
where each $M_i/M_{i-1}$ is $\alpha_i$-critical and $$\alpha_1\le\dots\le\alpha_k.$$ A critical series always exists for $M$ \cite[6.2.20]{MR}, and in practice is easy to write down. It is a particular case of \cite[Theorem~2.1]{Gull} that we then have
$$\ell(M)=\omega^{\alpha_k}+\dots+\omega^{\alpha_1};$$
if $n_\alpha$ is the number of $i$ such that $\alpha_i=\alpha$, then the family of non-negative integers $(n_\alpha)$ is finitely supported, and we can rewrite this formula as
$$\ell(M)=\sum_\alpha \omega^\alpha\cdot n_\alpha\quad\textnormal{(sum in reverse order)}.$$

In particular, we have

\begin{prop}\label{elc}
Let $A$ be a commutative noetherian ring, $P$ a prime ideal such that $A/P$ has Krull dimension $\alpha\ge 1$. Let $M$ be a torsion-free $A/P$-module of rank~$r$. Then $\ell_A(M)=\omega^{d}\cdot r$.
\end{prop}
\begin{proof}
Let $(M_i)$ be a critical series as above. We have $\alpha_i\le\alpha$ for all $i$. Since $M$ is torsion-free over $A/P$, so is $M_1$, so we have $\alpha_1=\alpha$. Therefore $\alpha_i=d$ for all $i$; since $M_i/M_{i-1}$ is a subquotient of $M$, this implies that it is torsion-free; since it is critical, it is torsion-free of rank one. In particular, $k=r$. The formula above gives $\ell(M)=\omega^d\cdot r$.
\end{proof}

\subsection{General lengths}

Let $H$ be a semigroup. An $H$-group is by definition a group $G$ endowed with an action of $H$ by group endomorphisms. 

\begin{exe}If $A$ is a ring and $M$ is an $A$-module, then $M$ is an $A$-group, where $A$ is viewed as a multiplicative semigroup.
\end{exe}

An $H$-group satisfies $H$-max-n if every non-decreasing sequence of $H$-stable normal subgroups of $G$ stabilizes. When $G$ is an $H$-group and $N$ an $H$-stable normal subgroup of $G$, the group $N$ has natural action of both $H$ and $G$, hence of the semidirect product $G\rtimes H$, which we denote by $GH$ for short.

Let $G$ be an $H$-group satisfying $H$-max-n. We can define by induction the length of $G$ as
$$\ell_H(G)=\sup(\ell_H(G/N)+1),$$ where $G/N$ ranges over all proper $H$-quotients of $G$. Here of course we assume $\sup(\emptyset)=0$, which gives $\ell_H(\{1\})=0$. 

\begin{lem}
For any ordinal $\beta\le\ell(M)$, there exists a $H$-quotient $G/N$ of $G$ such that $\ell_H(G/N)=\beta$.\label{convex}
\end{lem}
\begin{proof}
Let $\beta$ be the smallest counterexample, and $\gamma\le\alpha$ the smallest ordinal $>\beta$ such that $\ell_H(G/N)=\gamma$ for some $H$-quotient $G/N$. Then the definition of $\ell_H(G/N)$ leads to the existence of $N'$ such that $\beta\le\ell_H(G/N')<\gamma$, contradicting the definition of $\gamma$.
\end{proof}

Let $\mathcal{N}_H(G)$ denote the set of $H$-stable normal subgroups of $G$. Let $\cb_H(G)$ denote the Cantor-Bendixson rank of $\{1\}$ in $\mathcal{N}_H(G)$.

\begin{lem}
Let $G$ be an $H$-group satisfying $H$-max-n. Then $\mathcal{N}_H(G)$ is scattered and $\cb_H(G)\le\ell_H(G)$.\label{maxscat}
\end{lem}
\begin{proof}
This is a straightforward induction on $\ell_H(G)$.
\end{proof}

However this result is not optimal in general (see Proposition \ref{cblp}).

If $\alpha,\beta$ are ordinals, their natural sum is define inductively as follows \cite[XIV.28]{Sie}
$$\alpha\oplus\beta=\max\left(\sup_{\gamma<\alpha}((\gamma\oplus\beta)+1),\sup_{\gamma<\alpha}((\alpha\oplus\gamma)+1)\right),$$with $\sup\emptyset=0$.
Recall that any ordinal $\alpha$ has a unique {\it Cantor form} \cite[XIV.19]{Sie}
$$\alpha=\sum_\gamma \omega^\gamma\cdot n_\gamma,$$
where the sum is indexed by $\gamma$ ranging over the ordinal in decreasing order, and $(n_\gamma)$ is a finitely supported family of non-negative integers. If $\beta=\sum\omega^\gamma\cdot n'_\gamma$ is also in Cantor form, then their natural sum is
$$\alpha\oplus\beta=\sum_\gamma\omega^\gamma\cdot(n_\gamma+n'_\gamma).$$

The following lemma generalizes Theorem 2.1 and Proposition 2.11 in \cite{Gull}.

\begin{lem}
Let $G$ be an $H$-group with $H$-max-n, in an exact sequence of $H$-groups $$1\to M\to G\to Q\to 1.$$ Then $$\ell_H(Q)+\ell_{GH}(M)\;\le\;\ell_H(G)\;\le\;\ell_{H}(Q)\oplus\ell_{GH}(M).$$
When $G=M\times Q$ as an $H$-group, $$\ell_H(G)=\ell_H(Q)\oplus\ell_{GH}(M).$$\label{ext}
\end{lem}
\begin{proof}
All facts are directly obtained by induction on $\ell_H(G)$. For instance, let us check that $\ell_H(G)\le\ell_{H}(Q)\oplus\ell_{GH}(M)$. If $N$ is an $H$-invariant normal subgroup of $G$, then if $N\cap M$ is non-trivial, by induction $$\ell_H(G/(N\cap M))\le \ell_{H}(Q)\oplus\ell_{GH}(M/(N\cap M))<\ell_{H}(Q)\oplus\ell_{GH}(M)$$ by definition of the natural sum. If $N\cap M=\{1\}$ and the projection $p(N)$ of $N$ on $Q$ is non-trivial, then by induction $$\ell_H(G/N)\le \ell_{H}(Q/p(N))\oplus\ell_{GH}(M)<\ell_{H}(Q)\oplus\ell_{GH}(M)$$ again by definition of the natural sum. In all cases, we get $\ell_H(G/N)<\ell_{H}(Q)\oplus\ell_{GH}(M)$, so passing to the supremum we get $\ell_H(G)\le\ell_{H}(Q)\oplus\ell_{GH}(M)$.
\end{proof}

\subsection{Reduced length}\label{rl}

There is a well-defined notion of left Euclidean division for ordinals. In particular, if $\alpha$ is an ordinal, it is easy to check that there is a unique ordinal $\alpha'$ such that $\alpha=\omega\cdot\alpha'+r$ with $r<\omega$. For instance, $1'=0$, $(\omega^{n+1})'=\omega^n$ for $n<\omega$ and $(\omega^\alpha)'=\omega^\alpha$ for $\alpha\ge\omega$.

\begin{defn}
Let $G$ be a $H$-group satisfying $H$-max-n. Define $\ell'_H(G)$ as $(\ell_H(M))'$. 
\end{defn}

\begin{prop}
We have
$$\ell'_H(G)=\sup\{\ell'_H(G/N)+1\},$$ where $G/N$ ranges over all $H$-quotients of $G$ with $\ell_{GH}(N)\ge\omega$.
\end{prop}
\begin{proof}
Define $\ell''_H(G)$ by the inductive formula
$$\ell''_H(G)=\sup\{\ell''_H(G/N)+1\}.$$ Let us prove by induction on $\alpha=\ell_H(G)$ that $\ell'_H(G)=\ell''_H(G)$.
Let $N$ be a normal $H$-subgroup of $G$ with $\ell_{GH}(N)\ge\omega$. By Lemma \ref{ext}, we have 
$$\ell_H(G/N)+\ell_{GH}(N)\le\ell_H(G).$$
So 
$$\omega\cdot (\ell'_H(G/N)+1)\le\omega\cdot\ell'_H(G);$$
since we can ``simplify" by $\omega$ on the left, this gives, using the induction hypothesis
$$\ell''_H(G/N)+1\le\ell'_H(G);$$
taking the supremum over $N$, we get $\ell''_H(G)\le\ell'_H(G)$.

Conversely, take $\alpha<\ell'_H(G)$. So $\omega\cdot\alpha+\omega\le\ell_H(G)$. By Lemma \ref{convex}, there exists a normal $H$-subgroup $N$ of $G$ with $\ell_H(G/N)=\omega\cdot\alpha$. In particular, $\ell'_H(G/N)=\alpha$, so by induction hypothesis, $\ell''_H(G/N)=\alpha$. If we had $\ell_{GH}(N)<\omega$ then Lemma \ref{ext} would imply $\ell_H(G)<\omega\cdot\alpha+\omega$, a contradiction. Therefore $\ell''_H(G)>\alpha$. Since this holds for any $\alpha<\ell'_H(G)$, we deduce $\ell'_H(G)\le\ell''_H(G)$.
\end{proof}

\begin{lem}
Let $G$ be an $H$-group with $H$-max-n, in an exact sequence of $H$-groups $$1\to M\to G\to Q\to 1.$$ Then $$\ell'_H(Q)+\ell'_{GH}(M)\;\le\;\ell'_H(G)\;\le\;\ell'_{H}(Q)\oplus\ell'_{GH}(M).$$
When $G=M\times Q$ as an $H$-group, $$\ell'_H(G)=\ell'_H(Q)\oplus\ell'_{GH}(M).$$\label{extp}
\end{lem}
\begin{proof}
Since the sum and natural sum commute with $\alpha\mapsto\alpha'$, this immediately follows from Lemma \ref{ext}.
\end{proof}

\begin{lem}
Let $G$ be an $H$-group satisfying $H$-max-n. Then for every $H$-stable normal subgroup $F$ of $G$ with $\ell_{GH}(F)<\omega$, we have $\ell'_{H}(G/F)=\ell'_H(G)$.\label{quotfini}
\end{lem}
\begin{proof}
By definition of $\ell'_H$, we have $\ell'_{GH}(F)=0$. So this follows readily from Lemma~\ref{extp}.
\end{proof}

\begin{lem}
Let $G$ be an $H$-group satisfying $H$-max-n. Suppose that $G$ is residually finite as an $H$-group. Then
$$\ell_H(G)<\omega\Leftrightarrow\ell'_H(G)=0\Leftrightarrow G\textnormal{ is finite.}$$\label{finite}
\end{lem}
\begin{proof}
The left-hand equivalence is true by definition of $\ell'_H$, without assuming residual finiteness. Clearly $G$ finite implies $\ell_H(G)<\infty$. Conversely if $G$ is infinite,
it has a decreasing sequence $(M_n)$ of $H$-stable (finite index) normal subgroups, the sequence $\ell_H(G/M_n)$ is increasing and $\ell_H(M)\ge\omega$.
\end{proof}

\begin{defn}Using Lemma \ref{ext}, we can define, for every $H$-group $G$ with $H$-max-n, $E_H(G)$, [resp. $W_H(G)$], as its unique largest normal $H$-invariant subgroup $N$ with $\ell_H(N)<\omega$ [resp. $\ell'_H(N)<\omega$, i.e.~$\ell_H(N)<\omega^2$]. It follows from Lemmas \ref{ext}and \ref{extp} that $E_H(G/E_H(G))=\{1\}$ and $W_H(G/W_H(G))=\{1\}$.\end{defn}

\begin{prop}\label{cblp}
Let $G$ be an $H$-group satisfying $H$-max-n. Suppose that $G$ is residually finite as an $H$-group, as well as all its $H$-quotients. Then $\cb_H(G)=\ell'_H(G)$.
\end{prop}
\begin{proof}
Consider a counterexample $G$ with $\ell_H(G)=\alpha$ and assume that the lemma is proved for every $H$-group of $\ell_H<\alpha$.

By Lemma \ref{finite}, $M=E_H(G)$ is finite. Therefore if $N$ is close enough to $\{1\}$ we have $N\cap M=\{1\}$; if $N\neq\{1\}$ this forces $N$ to be infinite. In this case, by induction $\cb_H(G/N)=\ell'_H(G/N)<\ell'_H(G)$. Accordingly, $\cb_H(G)\le\ell'_H(G/N)+1$, so $\cb_H(G)\le\ell'_H(G)$.

Conversely, consider any ordinal $\beta<\alpha$ and a finite subset $I$ of $G-\{1\}$. Then $G$ has a proper quotient $G/N$ with $\ell_{GH}(N)\ge\omega$ (so $N$ is infinite) and $\ell'_H(G/N)=\beta$, and has an $H$-stable normal finite index subgroup $L$ with $L\cap I=\emptyset$; necessarily $N\cap F\neq\{1\}$.
As $L/(L\cap N)$ is finite, we have $\ell'_{GH}(L/(L\cap N))<\omega$. So by Lemma \ref{quotfini} $\ell'_H(G/(L\cap N))=\ell'_H(G/L)=\beta$ and $\cb_H(G/(L\cap N))=\beta$ again by induction. Thus every neighbourhood of $\{1\}$ in $\mathcal{N}_H(G)$ contains an element of Cantor-Bendixson $\beta$. As this holds for every $\beta<\alpha$, we get $\cb_H(G)\ge\alpha=\ell'_H(G)$. 
\end{proof}

Examples with $\ell'_H(G)<\cb_H(G)<\ell_H(G)$ were obtained in \cite[Lemma~15]{Cor}.

\begin{cor}\label{ccblp}
Under the same assumptions, $\CB(\mathcal{N}_H(G))=\cb_H(G)+1$.
\end{cor}
\begin{proof}
Set $\alpha=\cb_H(G)$. Then by Proposition \ref{cblp}, $\mathcal{N}(G)^{(\alpha)}$ is contained in the set of finite normal $H$-invariant subgroups of $G$ (it is actually equal in view of Lemma \ref{quotfini}) and contains $\{1\}$. The assumption $H$-max-n then  implies that the non-empty set $\mathcal{N}(G)^{(\alpha)}$ is finite, so $\mathcal{N}(G)^{(\alpha+1)}=\emptyset$ and $\CB(\mathcal{N}(G))=\alpha+1$.
\end{proof}

\begin{lem}Let $G$ be an $H$-group with $H$-max-n and $N$ an $H$-stable normal subgroup of $G$ contained in $E_H(G)$ (resp. $W_H(G)$). Then $\ell_H(G)=\ell_H(G/N)+\ell_{GH}(N)$ (resp. $\ell'_H(G)=\ell'_H(G/N)+\ell'_{GH}(N)$). Moreover, $\ell(G/E_H(G))$ and $\ell'(G/W_H(G))$ are not successor ordinals.\label{reducw}
\end{lem}
\begin{proof}
The first statement is a particular case of Lemmas \ref{ext} and \ref{extp}, since when $\alpha,\beta$ are ordinals with $\beta$ finite, $\alpha\oplus\beta=\alpha+\beta$.

If $\ell_H(G)=\alpha+1$, then by definition of $\ell_H$, for some non-trivial $H$-stable normal subgroup $N$, we have $\ell_H(G/N)=\alpha$. From the left-hand inequality in Lemma \ref{ext} we deduce $\ell_H(N)\le 1$, so $\ell_H(N)=1$, so $N\subset E_H(G)$ and $E_H(G)$ is non-trivial. Similarly if $\ell'_H(G)$ is a successor ordinal, $W_H(G)$ is non-trivial. This proves the second statement.
\end{proof}

Suppose that we have an extension of $H$-groups
$$1\to M\to G\to Q\to 1,$$
for which we want to compute $\ell'_H(G)$.

\begin{prop}
Let $G$ be an $H$-group with $H$-max-n lying in an extension
$$1 \to M\to G\to Q\to 1$$ with $\ell'_H(Q)<\infty$. Then $$\ell'_H(G)=\ell'_{GH}(M/W_{GH}(M))+\ell'_{GH}(W_H(G)).$$\label{ellprimedecompo} 
\end{prop}
\begin{proof}

By Lemma \ref{reducw}, $\ell'_H(G)=\ell'_H(G/W_H(G))+\ell'_{GH}(W_H(G))$. Now $W_H(G)\cap M=W_{GH}(M)$, we have an extension
$$1\to M/W_{GH}(M)\to G/W_H(G)\to Q'\to 1,$$
for some quotient $Q'$ of $Q$. 

We consider two cases.
\begin{itemize}
\item $M/W_{GH}(M)=\{1\}$. Then $G=W_H(G)$ and the lemma holds.
\item $M/W_{GH}(M)\neq\{1\}$. Then by Lemma \ref{reducw}, $\ell'_{GH}(M/W_{GH}(M))$ is a limit ordinal, so as $\ell'_H(Q')<\omega$, we get $$\ell'_H(Q')+\ell'_{GH}(M/W_{GH}(M))=\ell'_{GH}(M/W_{GH}(M)).$$ By Lemma \ref{extp}
$$\ell'_{GH}(M/W_{GH}(M))\le\ell'(G/W_H(G))\le\ell'_{GH}(M/W_{GH}(M))+\ell_H(Q');$$
and by Lemma \ref{reducw}, $\ell'(G/W_H(G))$ is a limit ordinal and $\ell'_H(Q')<\omega$, so we thus get $\ell'(G/W_H(G))=\ell'_{GH}(M/W_{GH}(M))$.\qedhere
\end{itemize}
\end{proof}

\begin{cor}Under the same hypotheses, if $M$ contains its own centralizer in $G$ and $W_{GH}(M)=\{1\}$, then $$\ell'_H(G)=\ell'_{GH}(M).$$\label{owncentralizer}\end{cor}

\subsection{Proof of Theorem \ref{cbme}}\label{pcb}

\begin{lem}\label{hirlp}
Let $G$ be an $H$-group with max-n. Then if finite, $\ell'_H(G)$ is the supremum of lengths $k$ of chains of $GH$-subgroups $1=N_0\subset\dots\subset N_k=G$ with $\ell'_H(N_i/N_{i-1})\ge 1$ for all $i$.
\end{lem}
\begin{proof}
If we have such a chain, then it follows from Lemma \ref{ext} that $\ell'_H(G)\ge k$. The converse is a trivial induction on $\ell'_H(G)$.
\end{proof}

Let us now prove Theorem \ref{cbme}.
It follows from the discussion in Paragraph \ref{lengthcr} that if $M$ is a noetherian module over a commutative ring, then $\dim(M)\le 1$ (Krull dimension) if and only if $\ell'(M)<\omega$. Therefore the definitions of $W(M)$ given in the introduction and in Paragraph \ref{rl} coincide.
Besides, we have

Suppose that $G$ is residually finite as well as its quotients and $N$ is a normal subgroup. By Lemma \ref{hirlp},if $\ell'G(N)<\omega$, then it coincides with $h_G(N)$ as defined in the introduction. Similarly, $W(G)$ coincides with $\textnormal{Hir}(G)$, also defined in the introduction. Also, $\cb(G)=\ell'(G)$ by Proposition \ref{cblp}.

Given these remarks, we see that (\ref{icb}) of the theorem appears as a particular case of Proposition \ref{ellprimedecompo}.

For (\ref{est}), the left-hand inequality is clear since $\ell'_G(M)\le\ell'_G(G)=\cb(G)$. 

Finally (\ref{pri}) is a consequence of (\ref{icb}). Indeed, $W(M)=\{0\}$ and $\ell'(M)=\omega^{d-1}\cdot r$ by Proposition \ref{elc}; moreover $\textnormal{Hir}(G)=\{1\}$. Indeed, since $W(M)=\{1\}$ ($M$ being now written multiplicatively), $\textnormal{Hir}(G)\cap M=\{1\}$. In particular, $\textnormal{Hir}(G)$ centralizes $M$. Since the action of $Q$ on $M$ is faithful, this implies that the projection of $\textnormal{Hir}(G)$ on $Q$ is trivial, hence $\textnormal{Hir}(G)\subset M$, hence $\textnormal{Hir}(G)=\{1\}$\qed.

\section{Some examples}\label{se}

\begin{prop}\label{hirsch}
If $G$ is a virtually polycyclic group, then $\ell'(G)\le h(G)$, the Hirsch length of $G$. The equality $\ell'(G)=h(G)$ holds if and only if $G$ is supersolvable (e.g. $G$ is nilpotent).
\end{prop}
\begin{proof}
Indeed, by Lemma \ref{finite}, if finite, $\ell'(G)$ is the greatest number of infinite subfactors in a {\it normal} series of $G$. On the other hand, $h(G)$ is the greatest number of subfactors in a {\it subnormal} series of $G$. To say that $G$ is supersolvable just means that there exists a normal series in which all infinite subfactors are cyclic, whence the equality. If $G$ is virtually polycyclic, there exists a normal series with exactly $\ell'(G)$ infinite subfactors, all torsion-free (at the cost of adding some finite subfactors in the normal series). If $G$ is not supersolvable, then one of these infinite subfactors has to have rank at least two, so $h(G)>\ell'(G)$. 
\end{proof}

\begin{exe}
If $G=\mathbf{Z}^k\rtimes F$ with $F$ finite, then $\ell'(G)$ is the number of irreducible representations in which $\mathbf{Q}^k$ decomposes under the action of $F$.
\end{exe}

\begin{prop}\label{wdwn}
For all $d\ge 0$, $n\ge 1$,
$$\ell'(\mathbf{Z}^k\wr\mathbf{Z}^d)=\omega^d\cdot k.$$
\end{prop}
\begin{proof}
The group $G=\mathbf{Z}^k\wr\mathbf{Z}^d$ can be written as $\mathbf{Z}[Q]^k\rtimes Q$ with $Q=\mathbf{Z}^d$. As $\mathbf{Z}[Q]$ is a domain of Krull dimension $d+1$, we have $\ell'_Q(\mathbf{Z}[Q]^k)=\omega^{d}\cdot k$ by Proposition \ref{elc}. Now, if $d\ge 1$, Corollary \ref{owncentralizer} applies to give $\ell'(G)=\omega^d\cdot k$. The case $d=0$ is a particular case of Proposition \ref{hirsch}.
\end{proof}

Denote by $\mathbf{C}_m$ the cyclic group of order $m$ and by $\delta(m)$ the total number factors in a prime decomposition of $m$ (e.g. $\delta(18)=3$).

\begin{prop}For all $d\ge 1$,
$$\ell'(\mathbf{C}_m\wr\mathbf{Z}^{d+1})=\omega^d\cdot \delta(m);\;\;\ell'(\mathbf{C}_m\wr\mathbf{Z})=\delta(m)+1.$$
\end{prop}
\begin{proof}

The group $G=C_m\wr\mathbf{Z}^{d+1}$ can be written as $\mathbf{Z}/m\mathbf{Z}[Q]\rtimes Q$. The $\mathbf{Z}[Q]$-module $\mathbf{Z}/ m\mathbf{Z}[Q]$ can be written as an iterated extension of $\delta(m)$ modules, each of the form $\mathbf{Z}/p\mathbf{Z}[Q]$. As the latter is a domain of Krull dimension $d+1$, it has $\ell'=\omega^{d}$. Now Lemma \ref{extp} implies that $\ell'(\mathbf{Z}/m\mathbf{Z}[Q])=\omega^d\cdot \delta(d)$.

If $Q=\mathbf{Z}$, then in the principal ideal domain $\mathbf{Z}/p\mathbf{Z}[Q]$, every nonzero ideal has finite index, so $\ell'_Q(\mathbf{Z}/p\mathbf{Z}[Q])=1$. So $\mathbf{Z}/m\mathbf{Z}[Q]$ has a normal series of length $\delta(m)$ in which each subfactor has $\ell'_Q=1$, so $\ell'_Q(\mathbf{Z}/m\mathbf{Z}[Q])=\delta(m)$ by Lemma \ref{extp}. Again by Lemma \ref{extp}, we deduce that $\ell'(G)=\delta(m)+1$.\end{proof}

\section{Actions of finite groups}\label{finiteaction}

Let $A$ be a ring and $G$ a group acting on $A$ by ring automorphisms. We call a $GA$-{\it module} an $A$-module endowed with a $G$-action by group automorphisms, such that, for all $g\in G$, $a\in A$ and $m\in M$, we have $g(am)=(ga)(gm)$.

A $GA$-{\it submodule} is the same as a $G$-invariant $A$-submodule. In particular, if $G$ is finite, a module is finitely generated, resp. Noetherian as an $A$-module if and only if it so as a $GA$-module.

Assume now that $M$ is a Noetherian $GA$-module and that $G$ is finite. We consider the length and reduced length as defined in Section \ref{length}, with $H$ the underlying multiplicative semigroup of $A$. As we do this all along this section, we drop the index $A$ on $\ell'$.

So the definitions of Section \ref{length} read as 
$$\ell_G'(M)=\sup\{\ell_G'(M/N)|N\text{ non-Artinian $GA$-submodule of }M\}$$
and
$$\ell'(M)=\sup\{\ell'(M/N)|N\text{ non-Artinian $A$-submodule of }M\},$$
as it was defined in \cite{Cor}. Clearly, $\ell'_G(M)\le\ell'(M)$.

Suppose that $A$ is Noetherian, and, to simplify the exposition, that it has finite Krull dimension. Let $M$ be a finitely generated $A$-module of Krull dimension $d\ge 1$.
In \cite{Cor} we showed that
$$\ell(M)=\omega^{d}\cdot\ell_d(M)+o(\omega^d)$$
where $o(\omega^d)$ denotes some ordinal $<\omega^d$ and
and $\ell_d(M)$ is a positive integer. Similarly define $\ell_{G,d}(M)$ so that $\ell(M)=\omega^{d}\cdot\ell_{G,d}(M)+o(\omega^d)$. Note that {\it a priori} $\ell_d(M)$ is a non-negative integer.

\begin{prop}
Suppose that $A$ is Noetherian of finite Krull dimension $d$, and $G$ is finite of order $n$. Let $M$ be a finitely generated $GA$-module, of Krull dimension $\le d$ (as an $A$-module). Then
$$\ell_{G,d}(M)\le\ell_d(M)\le n\ell_{G,d}(M).$$ 
\end{prop}
\proof
The left-hand inequality is an obvious consequence of $\ell_G(M)\le\ell(M)$. We prove the right-hand inequality by induction on $\ell(M)+d$. First, if $\ell_d(M)=0$ this is trivial. So we suppose $\ell_d(M)\ge 1$.

Since $M$ has Krull dimension $d$, there exists an associated prime ideal of coheight $d$, i.e.~a $A$-submodule $N$ of $M$ with $\ell(N)=\omega^d$. Let $N'$ be the $GA$-submodule generated by $N$: it is generated by the $gN$ for $g\in G$ and therefore is, as an $A$-module, a quotient of $N^n$. So $\ell_d(N')\le n\ell_d(N)=n$.
By \cite[Lemma 7]{Cor}, $\ell_d$ is additive on exact sequences of modules of Krull dimension $\le d$. If $\ell(N')<\ell(M)$, the inequality to proves holds for both $N'$ and $M/N'$, so by additivity holds for $M$.

So suppose $\ell(M)\le\ell(N')$. Then $\ell_d(M)\le\ell_d(N')\le n$. So we just have to prove that $\ell_{G,d}(M)\ge 1$, or equivalently that $\ell_G(M)\ge\omega^d$. If $d=0$ this is obvious. Otherwise, for any integer $m$, $M$ has a submodule $L$ with $\omega^{d-1}\cdot mn\le\ell(L)<\omega^d$. Replacing $L$ by $\bigcap_{g\in G}gL$ if necessary, we can suppose that $L$ is $G$-invariant. By induction hypothesis, $\ell_{G,d-1}(M/L)\ge m$. So $\ell_G(M)\ge\omega^{d-1}\cdot m$. Since this holds for any $m$, we deduce that $\ell_G(M)\ge\omega^d$.
\endproof

\begin{cor}
Under the same hypotheses, if $\ell'(M)=\omega^d$, then $\ell'_G(M)=\omega^d$.\label{omegad}
\end{cor}

\begin{cor}
Let $G$ be a finite group, $A$ be a finitely generated domain of Krull dimension $d\ge 1$ with a $G$-action, and $I$ a non-zero $G$-invariant ideal in $A$. Then $\ell'_{G}(I)=\omega^{d-1}$.\label{ideal}
\end{cor}
\begin{proof}
In view of Corollary \ref{omegad}, it is enough to check that $\ell'(I)=\omega^{d-1}$, which is a particular case of Proposition \ref{elc}. 
\end{proof}

The following lemma is well-known.

\begin{lem}
Let $A$ be a Noetherian ring and $M$ a finitely generated non-Artinian $GA$-module. Then $M$ is residually Artinian as an $A$-module.\label{ra}\end{lem}
\proof
Pick a non-zero element $x_0$ in $M$.
Let $W$ be a maximal $A$-submodule of $M$ not containing $x_0$. We claim that $M/W$ is Artinian.

We can suppose that $W=\{0\}$, i.e.~that $x_0$ is contained in every non-zero $A$-submodule of $M$ and we have to prove that $M$ is Artinian.

If $M$ is non-Artinian, then it has an associate ideal $P$ such that $A/P$ is not a field. So $M$ contains an $A$-submodule $N$ isomorphic to $A/Q$. Pick a non-zero non-invertible element $a$ in $A/P$. By a standard application of Artin-Rees lemma, we have $\bigcap_{n>0} a^nN=\{0\}$. By the assumption on $x_0$, we get $a^nN=\{0\}$ for some $n$, i.e. $a^n(A/P)=0$, and therefore as $P$ is prime, we obtain $a\in P$, a contradiction.
\endproof

\begin{lem}
Let $G$ be a finite group, $A$ be a Noetherian ring with a $G$-action, and $M$ a finitely generated non-Artinian $GA$-module. Then $M$ is residually Artinian as a $GA$-module.\label{raG}\end{lem}
\proof
Pick a non-zero element $x_0$ in $M$. Then there exists by Lemma \ref{ra} an $A$-submodule $N$ of $M$ such that $x_0\notin N$ and $M/N$ is Artinian.

If $N'=\bigcap_{g\in G}gN$, then $M/N'$ embeds into $\prod_{g\in G} M/gN$, so is Artinian as well. Moreover, $N'$ is a $GA$-submodule and $x_0\notin N'$.
\endproof

\begin{prop}
Let $G$ be a finite group, $A$ be a finitely generated ring with a $G$-action, and $M$ a finitely generated $GA$-module. 
Then the Cantor-Bendixson rank of $M$ as a $GA$-module, i.e.~the Cantor-Bendixson rank of $\{0\}$ in the set of $GA$-submodules of $M$, is $\ell'_G(M)$. 
\end{prop}
\proof
Any Artinian finitely generated $A$-module is finite: this is a classical consequence of the Nullstellensatz (see for instance \cite[Lemma 13]{Cor}). Therefore, using Lemma \ref{raG}, $M$ is residually finite as a $GA$-module. So the proposition appears as a particular case of Proposition \ref{cblp}.
\endproof


\section{The Bieri-Strebel invariant and tensor products}\label{bs}

In all this section, we consider a finitely generated metabelian group $G$, inside an extension
$$1\to M\to G\to Q\to 1,$$
with $Q$ abelian and $M$ abelian. So $M$ is a finitely generated $\mathbf{Z}[Q]$-module.

If $v\in\Hom(Q,\mathbf{R})$, we set $Q_v=\{q\in Q|v(q)\ge 0\}$. We set
$$\Gamma(M)=\{v\in \textnormal{Hom}(Q,\mathbf{R})|M\text{ is not a finitely generated $Q_v$-module}\}\cup\{0\}.$$
This is a closed subset \cite[Proposition 2.2]{BS} of the vector space $\Hom(Q,\mathbf{R})$, and is further studied in \cite{BS2}.
We write $\Gamma^{\pm}(M)=\Gamma(M)\cap(-\Gamma(M))$.

\begin{thm}[Bieri-Strebel \cite{BS}]
The finitely generated metabelian group $G$ is finitely presented if and only if $$\Gamma^\pm(M)=\{0\}.$$\label{bist}
\end{thm}

\begin{lem}
Fix $v\in \textnormal{Hom}(Q,\mathbf{R})-\{0\}$. Let $V$ be the $Q_v$-submodule generated by some finite generating subset of the $Q$-module $M$. Then we have the equivalences
\begin{itemize}
\item $v\notin\Gamma(M)$;
\item $qV=V$ for every $q\in Q$;
\item $qV\subset V$ for some $q$ with $v(q)<0$.
\end{itemize}\label{vmagique}
\end{lem}
\begin{proof}
Let us first check that the two first assertions are equivalent.
Suppose that $qV\neq V$ for some $q\in Q$. Replacing $q$ by $q^{-1}$ if necessary we can suppose that $v(q)\ge 0$. So $qV\subset V$, and we deduce that the sequence $(q^{-n}V)$ of $Q_v$-submodules of $M$ is strictly increasing, so that $M$ is not noetherian, hence not finitely generated as a $Q_v$-module, i.e. $v\in\Gamma(V)$.

Conversely the assumption $qV=V$ for all $q\in Q$ clearly implies that $V$ is a $Q$-module, hence $V=M$, so $M$ is a finitely generated $Q_v$-module.

The second assertion clearly implies the third, and the converse holds because the set of $q$ satisfying $qV\subset V$ is a subsemigroup, and clearly $Q$ is generated as a subsemigroup by $Q_v\subset\{q\}$ whenever $q\notin Q_v$. So $qV\subset V$ for all $q\in Q$, and multiplying by $q^{-1}$ we get $V\subset q^{-1}V$ for all $q\in Q$, so as $Q$ is closed under inversion, $qV=V$ for all $q\in Q$.
\end{proof}

Suppose that $Q=Q_1\times Q_2$, and let $A_i$ be the ring generated by $Q_i$. Suppose $M=M_1\otimes_\mathbf{Z} M_2$, where $M_i$ is a finitely generated $A_i$-module, and $M$ is naturally viewed as a $Q$-module. We have the identification $\Hom(Q,\mathbf{R})=\Hom(Q_1,\mathbf{R})\times\Hom(Q_2,\mathbf{R})$.

\begin{lem}
We have the inclusion
$$\Gamma(M)\subset\Gamma(M_1)\times\Gamma(M_2)$$
\end{lem}

\begin{proof}
Suppose that $(v_1,v_2)\notin\Gamma(M_1)\times\Gamma(M_2)$, say $v_1\notin\Gamma(M_1)$. Consider $V_i\subset M_i$ as in Lemma \ref{vmagique}. So there exists $q_1\in Q_1$ with $v_1(q_1)<0$ and $qV_1=V_1$. So $v(q_1,1)=v_1(q_1)<0$ and $(q_1,1)(V_1\otimes V_2)\subset (qV_1\otimes V_2)=V_1\otimes V_2$. So the $Q_v$-submodule $V$ generated by $V_1\otimes V_2$ is $(q_1,1)$-stable, hence it is a $Q$-submodule, so $v\notin\Gamma(M)$.
\end{proof}

\begin{cor}
If $\Gamma^\pm(M_1)=\{0\}$ and $\Gamma^\pm(M_2)=\{0\}$ then $\Gamma^\pm(M)=\{0\}$.\label{cortensor}
\end{cor}

Here is a classical example (below the coefficient ring $\mathbf{Z}$ can be replaced by $\mathbf{Z}/k\mathbf{Z}$).
\begin{lem}
Suppose that $A=M=\mathbf{Z}[u,(u+u^2)^{-1}]$, $Q=\mathbf{Z}^2$ acts by $(m,n)\cdot P(u)=P(u)u^m(1+u)^n$. Then $\Gamma(M)=\mathbf{R}_+(1,0)\cup\mathbf{R}_+(0,1)\cup\mathbf{R}_+(-1,-1)$. In particular, $\Gamma^\pm(M)=\{0\}$.
\end{lem}
\begin{proof}
First observe that $A=M$ has a ring automorphism of order three given by $u\mapsto -(1+u)/u\mapsto -1/(u+1)\mapsto u$. This implies that $\Gamma(M)$ is invariant under the matrix $\begin{pmatrix}-1 & 1 \\ -1 & 0\end{pmatrix}$ of order three, which rotates $(0,1)\mapsto (1,0)\mapsto (-1,-1)$. So it is enough to check that $(0,1)$ belongs to $\Gamma(M)$, but not $(a,b)$ if $a,b>0$.

If $v(m,n)=n$, then $\mathbf{Z}[Q_v]$ consists exactly of $\mathbf{Z}[u,u^{-1}]$. So $(0,1)\in\Gamma(M)$.

Consider $v(m,n)=am+bn$ with $a,b>0$.
Let $V$ be the $\mathbf{Z}[Q_v]$-submodule of $\mathbf{Z}[Q]$ generated by $1$. This is clearly a ring, so we just have to check that it contains $u,u^{-1},(1+u)^{-1}$.
Since $a\ge 0$, $u\in V$. Therefore $\mathbf{Z}[u]\subset V$.
Since $a>0$, we know that $v(n,-1)>0$ for large $n$, so $V$ contains $(-u)^n/(1+u)$ for large $n$, which can be written as $(1-(1+u))^n/(1+u)=P_1(u)+1/(1+u)$ with $P_1(u)\in\mathbf{Z}[u]$. So $1/(1+u)\in V$. As $b>0$, $V$ contains $(1+u)^n/u$ for large $n$. Since we can write $(1+u)^n/u=P_2(u)+1/u$ with $P_2(u)\in\mathbf{Z}[u]$, we deduce that $V$ contains $u^{-1}$.
\end{proof}

As $M$ is the tensor product of $k$ copies of $\mathbf{Z}[u,(u^2+u)^{-1}]$, from Corollary \ref{cortensor} we get
\begin{cor}\label{fpanneau}
If $M=A=\mathbf{Z}[u_1,\dots,u_k,s^{-1}]$ where $s=\prod_{i=1}^k(u_i^2+u_i)$, then $\Gamma^\pm(M)=\{0\}$.
\end{cor}

We will also need the following easy consequence of Theorem \ref{bist}.

\begin{lem}
Let $G$ be a finitely presented metabelian group in an exact sequence
$$1\to M\to G\to Q\to 1$$
with $M$ and $Q$ abelian. Let $H$ be a subgroup of $G$ whose projection on $Q$ is surjective (i.e. $HM=G$). Then $H$ is finitely presented as well.\label{bigsubgroup}
\end{lem}
\begin{proof}
By assumption we have an exact sequence
$$1\to M\cap H\to H\to Q\to 1.$$
So $M\cap H$ is a $Q$-submodule of $M$, hence is finitely generated as a $Q$-module, so $H$ is finitely generated. Next, we see that $\Gamma(M\cap H)\subset \Gamma(H)$ (this uses the fact that the rings $\mathbf{Z} Q_v$ implied in the definition of the Bieri-Strebel invariant, are noetherian, as localizations of polynomial rings, although they may be infinitely generated). So Theorem \ref{bist} implies that $H$ is finitely presented.
\end{proof}


\section{Free and split metabelian groups}\label{fsme}

\subsection{Free metabelian groups}\label{frme}

Let $\mathbf{FM}_d=\langle x_1,\dots,x_d\rangle$ denote the free metabelian group on $d$ generators. 
Consider the extension $$ 1 \to M\to \mathbf{FM}_d\to Q\to 1$$
with $Q\simeq\mathbf{Z}^d$ and $M=[\mathbf{FM}_d,\mathbf{FM}_d]$.

\begin{prop}\label{fmd}
As a $\mathbf{Z}[Q]$-module, $M$ is torsion-free of rank $d-1$.
\end{prop}
\begin{proof}
The {\it Magnus embedding} of $\mathbf{FM}_d$ is the following. We consider matrices
$$\begin{pmatrix}
  t & m  \\
  0 & 1  \\
  \end{pmatrix}$$
with $t\in Q$ and $m$ in $N$, the free $\mathbf{Z}[Q]$-module of rank $d$ with basis $(e_i)$.
The Magnus embedding $i$ is given by
$$x_i\mapsto \begin{pmatrix}
  u_i & e_i  \\
  0 & 1  \\
  \end{pmatrix}.$$
This is a well-defined map whose injectivity is due to Magnus \cite{Mag}. In particular, $M$ embeds as a $\mathbf{Z}[Q]$-module into $N$, so is torsion-free. Let $r$ denote its rank.
Denote by $N_0$ the $A$-submodule of $N$ consisting of all $\sum a_ie_i$ $(a_i\in A)$ satisfying $\sum_i(1-u_i)a_i=0$.

\begin{lem}
We have $M\subset N_0$. In particular $r\le d-1$.\label{inters}
\end{lem}
\begin{proof}
Write $[x,y]=x^{-1}y^{-1}xy$ and $x^y=y^{-1}xy$. For convenience we identify $\mathbf{FM}_d$ to its image by $i$. It is enough to prove that $[t,v]\in N_0$ for all $t,v\in \mathbf{FM}_d$.

We have, in any group, the equality $[t_1t_2,v]=[t_1,v]^{t_2}.[t_2,v]$. As $N_0$ is an $A$-submodule, it follows that for every $v$, the set of $t$ such that
$[t,v]\in N_0$ is closed under multiplication, and similarly it is closed under inversion, hence is a subgroup. The analog fixing $t$ also holds. So it is enough to check $[t,v]\in N_0$ for $t,v$ ranging over group generators. A computation gives
$$[x_i,x_j]=u_i^{-1}u_j^{-1}((1-u_j)e_i-(1-u_i)e_j),$$
which belongs to $N_0$.
\end{proof}

Besides, we have $r\ge d-1$. Indeed, as we just mentioned, for $j>1$ we have $$u_iu_j[x_1,x_j]=(1-u_j)e_1-(1-u_1)e_j;$$
this is a $\mathbf{Z}[Q]$-free family of cardinality $d-1$.
\end{proof}

\begin{thm}
For every $d\ge 2$ we have
$$\ell'(\mathbf{FM}_d)=\omega^{d}\cdot (d-1).$$
Moreover, for any proper quotient $G$ of $\mathbf{FM}_d$, we have $\ell'(G)<\omega^{d}\cdot (d-1)$.\label{fm}
\end{thm}
\begin{proof}
As the Krull dimension of $\mathbf{Z}[Q]$ is $\ge 2$ and $M$ is torsion-free as a $\mathbf{Z}[Q]$-module, Corollary \ref{owncentralizer} implies that $\ell'(\mathbf{FM}_d)=\ell'(M)$, and $\ell'(M)=\omega^d\cdot (d-1)$ by Propositions \ref{fmd} and \ref{elc}.
\end{proof}

In view of Proposition \ref{cblp}, Theorem \ref{fm} implies Theorem \ref{thmbounds}(\ref{melib}).


\subsection{Split metabelian groups}\label{splitm}

In this section, we prove Theorem \ref{thmbounds}(\ref{itemsplit}).

\begin{prop}
Fix $d\ge 1$. Let $G$ be a $d$-generated metabelian group in a {\em split} exact sequence
$$1\to M\to G\to Q\to 1.$$
Then $\ell'(G)<\omega^{d}$.
\end{prop}
\begin{proof}Assume that $\ell'(G)\ge\omega^d$. If the $\mathbf{Q}$-rank of $Q$ is less than $d$, then $\ell'(G)<\omega^{d}$ holds (even if the exact sequence is not split). So $Q$ is free abelian of rank $d$; in particular, $M=[G,G]$. For the same reason, the Krull dimension of $M$ has to be equal to $d+1$. Modding out by its torsion submodule, we can assume that $M$ is a nonzero torsion-free $\mathbf{Z}[Q]$-module.

Now, given a splitting, write the generators as $m_ie_i$, with $e_i\in Q$ ($i$th basis vector) and $m_i\in M$. The argument in the proof of Lemma \ref{inters} shows that $[G,G]$ is contained in the $Q$-submodule generated by the elements
$$(1-u_j)m_i-(1-u_i)m_j,$$
where $u_i$ is the indeterminate in $\mathbf{Z}[Q]$ corresponding to $e_i$. Since $[G,G]=M$, we deduce that $M=IM$, where $I$ is the ideal generated by all $1-u_i$. As this is a proper ideal and $M$ is torsion-free finitely generated, Nakayama's Lemma implies that $M=\{0\}$, a contradiction.
\end{proof}

\begin{rem}
This upper bound works more generally for the slightly broader class of finitely generated (metabelian) groups having two abelian subgroups $Q,M$ with $M$ normal, such that $G=MQ$. This class has the additional advantage to be stable under quotients, and any element $G=MQ$ in this class is actually a quotient of a finitely generated split metabelian group, namely $M\rtimes Q$.
\end{rem}

Let us now prove that the bound given in Theorem \ref{thmbounds}(\ref{itemsplit}) is sharp. Continue with $Q$ free of rank $d$ as above, assume $d\ge 2$, and define the ring
$$A_n=\mathbf{Z}[Q]/(2-x_2)^n$$

Consider the semidirect product $G_n=A_n\rtimes Q$. Define $m_i\in A_n$ with $m_1=1$, $m_2=0$, and any $m_i$ for $i\ge 3$.

\begin{lem}\label{gndg}
The group $G_n$ is generated by the family $(m_ie_i)_{1\le i\le d}$;
\end{lem}
\begin{proof}
Let $H$ be the group generated by this family, and set $N=H\cap M$, which is an ideal of $A_n$. It contains in particular $u_1u_2[m_1e_1,m_2e_2]=(1-u_2)$. As $1+(1-u_2)$ is nilpotent by construction, $1-u_2$ is invertible (using a formal series), so $N$ contains the element $1$ of $H$, hence $N=A_n$. Therefore $H=G_n$.
\end{proof}

\begin{lem}\label{lprian}
We have $\ell'(A_n)=\omega^{d-1}\cdot n$.
\end{lem}
\begin{proof}
Using Lemma \ref{extp}, it is enough to check that $$\ell'\left((2-x_2)^kA_n/(2-x_2)^{k+1}A_n\right)=\omega^{d-1}.$$ As $A_n$ is a domain, the $\mathbf{Z}[Q]$-module $(2-x_2)^kA_n/(2-x_2)^{k+1}A_n$ is isomorphic to $A_n/(2-x_2)A_n$, which is the domain of Laurent polynomials in $d-1$ variables over $\mathbf{Z}[1/2]$, so by Proposition \ref{elc}, $\ell'(A_n/(2-x_2)A_n)=\omega^{d-1}$ as expected.
\end{proof}

From Lemma \ref{gndg}, Lemma \ref{lprian} and Corollary \ref{owncentralizer}, we deduce

\begin{prop}
For every $d\ge 2$, the split metabelian group $G_n$ is $d$-generated and $\ell'(G_n)=\omega^{d-1}\cdot n$.\qed
\end{prop}


\section{Wreath products}\label{secw}

Proposition \ref{wreathcond} is a particular case of the following more general result.
Let $H,G$ be any finitely generated groups, and $X$ a $G$-set with finitely many orbits. Then the permutational wreath product $H\wr_X G$, which is defined as the semidirect product $H^{(X)}\rtimes G$ (with the shifting action), is finitely generated.

\begin{prop}
Assume that the diagonal action of $G$ on $X^2$ has infinitely many orbits, and that $H\neq\{1\}$. Then $\cb^e(H\wr_X G)=\mathfrak{C}$.
\end{prop}
\begin{proof}
We assume for the sake of simplicity that $X$ is $G$-transitive; the extension of the proof to the general case is left as an exercise.

So we can write $X=G/L$. Set $\Gamma=H\wr_X G$. Consider the finitely generated group $S$ presented as 
$$\langle H,G|[H,L]\rangle.$$
This group is finitely generated and possesses $\Gamma$ as a quotient in a natural way. We are going to topologically embed a Cantor set into the set of quotients of $S$, so that the image contains $\Gamma$, which will imply that $\cb^e(\Gamma)=\mathfrak{C}$.

Consider an infinite subset $J$ of $G-L$ such that for any distinct $g,h\in J$, $g,g^{-1}\notin LhL$.
If $I$ is any subset of $J$, define $\Gamma_I$ as the quotient of $S$ by all $[H,gHg^{-1}]$ for all $g\notin L$ such that $LgL\cap (I\cup I^{-1})=\emptyset$. Then from \cite[Lemma 2.3]{CorGD} we deduce that for any $g\in J$, we have $[H,gHg^{-1}]=\{1\}$ if and only if $g\in I$. Therefore the map $I\to\Gamma_I$ is injective, so it embeds a Cantor set into the set of quotients of $S$, identified with $\mathcal{N}(S)$, and maps in particular $\emptyset$ to $\Gamma$. We claim that this map is continuous at $\emptyset$. Indeed, let $I_n\to\emptyset$. Let $g$ be a relation in $\Gamma$. Then $g$ is a consequence of finitely many relators, so $g=1$ in $\Gamma_{J-F}$ for some finite subset $F$ of $J$. As $I_n\to\emptyset$, eventually $I_n\cap F=\emptyset$, so $\Gamma_{I_n}$ is a quotient of $\Gamma_{J-F}$, so $g=1$ in $\Gamma_{I_n}$. Thus $\Gamma$ is a condensation point.
\end{proof}



\begin{thebibliography}{KM98b}

\bibitem[Abe]{Abe}
H.~Abels.
\newblock An example of a finitely presented solvable group.
\newblock In {\em Homological group theory (Proc. Sympos., Durham, 1977)},
  volume~36 of {\em London Math. Soc. Lecture Note Ser.}, pages 205--211.
  Cambridge Univ. Press, Cambridge, 1979.

\bibitem[Bass]{Bass} H. {\sc Bass}. {\em Descending chains and the Krull ordinal of commutative Noetherian rings}. J. Pure Appl. Algebra {\bf 1} (1971) 347--360.

\bibitem[BSo]{BSo} G.
Baumslag, D. Solitar. Some two-generator one-relator non-Hopfian groups. Bull. Amer. Math. Soc. 68 199--201 (1962).

\bibitem[BSt1]{BS} R. Bieri, R. Strebel. Valuations and finitely presented metabelian groups. Proc. London Math. Soc. (3) 41 (1980) 439--464.

\bibitem[BSt2]{BS2} R. Bieri, R. Strebel. A geometric invariant for modules over an abelian group.
J. Reine Angew. Math. 322 (1981), 170--189. 

\bibitem[Chab]{Chab} C. Chabauty.
\newblock Limite d'ensembles et g\'eom\'etrie des nombres.
\newblock Bull. Soc. Math. France 78 (1950), 143--151.

\bibitem[Cham]{Cham}
C.~Champetier.
\newblock L'espace des groupes de type fini.
\newblock Topology, 39(4):657--680, 2000.

\bibitem[CG]{CG}
C.~Champetier and V.~Guirardel.
\newblock Limit groups as limits of free groups.
\newblock Israel J. Math., 146:1--75, 2005.

\bibitem[Co1]{CorGD} Y. Cornulier. Finitely presented wreath products and double coset decompositions. 
Geom. Dedicata (2006) 122, 89--108.

\bibitem[Co2]{Cor} Y. Cornulier. The space of finitely generated rings. Internat. J. Algebra Comput. 19(3) (2009) 373--382.

\bibitem[CGP]{CGP}
Y.~de~Cornulier, L.~Guyot, and W.~Pitsch.
\newblock On the isolated points in the space of groups.
\newblock J. Algebra, 307(1):254--277, 2007.

\bibitem[Gri]{Gri}
R.~Grigorchuk.
\newblock Degrees of growth of finitely generated groups and the theory of
  invariant means.
\newblock {\em Izv. Akad. Nauk SSSR Ser. Mat.}, 48(5):939--985, 1984.

\bibitem[Gull]{Gull} H. Gulliksen. \newblock A theory of length for Noetherian modules. \newblock J. Pure Appl. Algebra {\bf 3} (1973) 159-170.

\bibitem[Hal1]{Hall} P. Hall. Finiteness conditions for soluble groups. Proc. London Math. Soc. i, No. 16, 419--436 (1954).

\bibitem[Hal2]{Hall2} P. Hall. On the finiteness of certain soluble groups. Proc. London Math. Soc. 9(3), 595--622 (1959).

\bibitem[Mag]{Mag} W. Magnus. On a theorem of Marshall Hall. Annals of Math. 40(4) (1939) 764--768.

\bibitem[MR]{MR} J. McConnell, J. Robson. Noncommutative noetherian rings. Grad. Stud. Math. vol. 30, Amer. Math. Soc. Providence, 2001. 

\bibitem[Ols]{Ols} A. Olshanskii.
On residualing homomorphisms and G-subgroups of hyperbolic groups. 
Internat. J. Algebra Comput. 3 (1993), no. 4, 365--409. 

\bibitem[Ros]{Ros} J. Roseblade. Group rings of polycyclic groups. J. Pure Appl. Algebra 3, 307--328 (1973).

\bibitem[Ser]{Serre} J-P. Serre. \newblock Cours d'arithm\'etique. \newblock Presses Universitaires de France, 1970.

\bibitem[Sie]{Sie} W. Sierpi\'nski. Cardinal and ordinal numbers. Second revised edition. Monografie Matematyczne, Vol. 34 Pa\'nstowe Wydawnictwo Naukowe, Warsaw, 1965.

\end{thebibliography}
\end{document}